\documentclass[11pt,twoside]{amsart}
\usepackage[latin1]{inputenc}
\usepackage[T1]{fontenc}
\usepackage{amsmath}
\usepackage{amsthm}
\usepackage{amssymb}
\usepackage[all]{xypic}

\newtheorem{thm}{Theorem}[section]

\newtheorem{conj}[thm]{Conjecture}

\numberwithin{equation}{section}

\theoremstyle{definition}

\newtheorem{remark}[thm]{Remark}
\newtheorem{ex}[thm]{Example}

\newcommand{\Db}{{\rm D}^{\rm b}}

\newcommand{\Aut}{{\rm Aut}}

\newcommand{\CH}{{\rm CH}}
\newcommand{\NS}{{\rm NS}}
\newcommand{\Pic}{{\rm Pic}}

\newcommand{\rk}{{\rm rk}}
\newcommand{\coh}{{\rm{Coh}}}

\newcommand{\comp}{\circ}

\newcommand{\Ext}{{\rm Ext}}

\newcommand{\cal}{\mathcal}

\newcommand{\kc}{{\cal C}}

\newcommand{\ke}{{\cal E}}

\newcommand{\kh}{{\cal H}}

\newcommand{\kk}{{\cal K}}

\newcommand{\ko}{{\cal O}}
\newcommand{\kp}{{\cal P}}

\newcommand{\ZZ}{\mathbb{Z}}
\newcommand{\QQ}{\mathbb{Q}}
\newcommand{\RR}{\mathbb{R}}
\newcommand{\CC}{\mathbb{C}}

\newcommand{\PP}{\mathbb{P}}
\newcommand{\OO}{{\rm O}}

\newcommand{\ch}{\rm{ch}}
\newcommand{\td}{\rm{td}}
\renewcommand{\to}{\xymatrix@1@=15pt{\ar[r]&}}
\renewcommand{\rightarrow}{\xymatrix@1@=15pt{\ar[r]&}}
\renewcommand{\mapsto}{\xymatrix@1@=15pt{\ar@{|->}[r]&}}
\renewcommand{\twoheadrightarrow}{\xymatrix@1@=15pt{\ar@{->>}[r]&}}
\renewcommand{\hookrightarrow}{\xymatrix@1@=15pt{\ar@{^(->}[r]&}}
\newcommand{\congpf}{\xymatrix@1@=15pt{\ar[r]^-\sim&}}
\renewcommand{\cong}{\simeq}

\begin{document}

\title[Chow groups and derived categories of K3 surfaces]{Chow groups and derived categories of K3 surfaces}
\author[D.\ Huybrechts]{Daniel Huybrechts}

\address{Mathematisches Institut,
Universit{\"a}t Bonn, Endenicher Allee 60, 53115 Bonn, Germany}
\email{huybrech@math.uni-bonn.de}

\maketitle


\begin{abstract}\noindent
The geometry of a K3 surface (over $\CC$ or over $\bar\QQ$) is
reflected by its Chow group and its bounded derived category of
coherent sheaves in different ways. The Chow group can be infinite
dimensional over $\CC$ (Mumford) and is expected to inject into
cohomology over $\bar\QQ$ (Bloch--Beilinson). The derived category
is difficult to describe explicitly, but its group of
autoequivalences can be studied by means of the natural
representation on cohomology. Conjecturally  (Bridgeland) the
kernel of this representation is generated by squares of spherical
twists. The action of these spherical twists on the Chow ring can
be determined explicitly by relating it to the natural subring
introduced by Beauville and Voisin.
\end{abstract}

\section{Introduction}

In algebraic geometry a K3 surface is a smooth projective surface
$X$ over a fixed field $K$ with trivial canonical bundle
$\omega_X\cong\Omega_X^2$ and $H^1(X,\ko_X)=0$. For us the field
$K$ will be either a number field, the field of algebraic numbers
$\bar\QQ$ or the complex number field $\CC$. Non-projective K3
surfaces play a central role in the theory  of K3 surfaces and for
some of the results that will be discussed in this text in
particular, but here we will not discuss those more analytical
aspects.

An explicit example of a K3 surface is provided by the Fermat
quartic in $\PP^3$ given as the zero set of the polynomial
$x_0^4+\ldots+x_3^4$. Kummer surfaces, i.e.\ minimal resolutions
of the quotient of abelian surfaces by the sign involution, and
elliptic K3 surfaces form other important classes of examples.
Most of the results and questions that will  be mentioned do not
loose any of their interest when considered for one of theses
classes of examples or any other  particular K3 surface.

This text deals with three objects naturally associated with any
K3 surface $X$: $$\Db(X),\phantom{i}\CH^*(X)\phantom{i}{\rm
and}\phantom{i} H^*(X,\ZZ).$$

If $X$ is defined over $\CC$, its \emph{singular cohomology}
$H^*(X,\ZZ)$ is endowed with the intersection pairing and a
natural Hodge structure. The \emph{Chow group} $\CH^*(X)$ of $X$,
defined over an arbitrary field, is a graded ring that encodes
much of the algebraic geometry. The \emph{bounded derived
category} $\Db(X)$, a linear triangulated category, is a more
complicated invariant and in general difficult to control.

As we will see, all three objects, $H^*(X,\ZZ)$, $\CH^*(X)$, and
$\Db(X)$ are related to each other. On the one hand, $\widetilde
H(X,\ZZ)$ as the easiest of the three can be used to capture some
of the features of the other two. But on the other hand and maybe
a little surprising, one can deduce from the more rigid structure
of $\Db(X)$ as a linear triangulated category interesting
information about cycles on $X$, i.e.\ about some aspects of
$\CH^*(X)$.

This text is based on  my talk at the conference `Classical
algebraic geometry today' at the MSRI in January 2009 and is meant
as a non-technical introduction to the standard techniques in the
area. At the same time it  surveys  recent developments and
presents some new results on a question on symplectomorphisms that
was raised in this talk (see Section \ref{sect:Aut}). I wish to
thank the organizers for the invitation to a very stimulating
conference.


\section{Cohomology of K3 surfaces}

The second singular cohomology of a complex K3 surface is endowed
with the additional structure of a weight two Hodge structure and
the intersection pairing. The Global Torelli theorem shows that it
determines the K3 surface uniquely. We briefly recall the main
features of this Hodge structure and of its extension to the Mukai
lattice which governs the derived category of the K3 surface. For
the general theory of complex K3 surfaces see e.g.\ \cite{BHPV} or
\cite{Ast}. In this section  all K3 surfaces are defined over
$\CC$.

\subsection{}  To any complex K3 surface
$X$ we can associate the singular cohomology $H^*(X,\ZZ)$ (of the
underlying complex or topological manifold). Clearly,
$H^0(X,\ZZ)\cong H^4(X,\ZZ)\cong\ZZ$. Hodge decomposition yields
$H^1(X,\CC)\cong H^{1,0}(X)\oplus H^{0,1}(X)=0$, since by
assumption $H^{0,1}(X)\cong H^1(X,\ko_X)=0$, and hence
$H^1(X,\ZZ)=0$. One can also show  $H^3(X,\ZZ)=0$. Thus, the only
interesting cohomology group is $H^2(X,\ZZ)$ which together with
the intersection pairing is abstractly isomorphic to the unique
even unimodular  lattice of signature $(3,19)$ given by $U^{\oplus
3}\oplus E_8(-1)^{\oplus2}$. Here, $U$ is the hyperbolic plane and
$E_8(-1)$ is the standard root lattice  $E_8$ changed by a sign.
Thus, the full cohomology $H^*(X,\ZZ)$ endowed with the
intersection pairing is isomorphic to $U^{\oplus 4}\oplus
E_8(-1)^{\oplus 2}$.

For later use we introduce $\widetilde H(X,\ZZ)$, which denotes
$H^*(X,\ZZ)$ with the Mukai paring, i.e.\ with a sign change in
the pairing between $H^0$ and $H^4$. Note that as abstract
lattices $H^*(X,\ZZ)$ and $\widetilde H(X,\ZZ)$ are isomorphic.

\subsection{}\label{subsect:GT} The complex structure of the K3 surface $X$ induces
a weight two Hodge structure on $H^2(X,\ZZ)$ given explicitly by
the decomposition $H^2(X,\CC)=H^{2,0}(X)\oplus H^{1,1}(X)\oplus
H^{0,2}(X)$. It is determined by the complex line
$H^{2,0}(X)\subset H^2(X,\CC)$ which is spanned by a trivializing
section of $\omega_X$ and by requiring the decomposition to be
orthogonal with respect to the intersection pairing. This natural
Hodge structure induces at the same time a weight two Hodge
structure on the Mukai lattice $\widetilde H(X,\ZZ)$ by setting
$\widetilde H^{2,0}(X)= H^{2,0}(X)$ and requiring $(H^0\oplus
H^4)(X,\CC)\subset \widetilde H^{1,1}(X)$.

The Global Torelli theorem and its derived version, due to
Piatetski-Shapiro and  Shafarevich resp.\ Mukai and Orlov, can be
stated as follows. For  complex projective K3 surfaces $X$ and
$X'$ one has:

 i) There exists an isomorphism $X\cong X'$ (over
$\CC$) if and only if there exists an isometry of Hodge structures
$H^2(X,\ZZ)\cong H^2(X',\ZZ)$.

 ii) There exists a $\CC$-linear
exact equivalence $\Db(X)\cong\Db(X')$ if and only if there exists
an isometry of Hodge structures $\widetilde
H(X,\ZZ)\cong\widetilde H(X',\ZZ)$.

\smallskip

 Note that for purely lattice theoretical reasons the
weight two Hodge structures $\widetilde H(X,\ZZ)$ and $\widetilde
H(X',\ZZ)$ are isometric if and only if their transcendental parts
(see \ref{subsect:NST}) are.


\subsection{}\label{subsect:NST} The Hodge index theorem shows that the intersection
pairing on $H^{1,1}(X,\RR)$ has signature $(1,19)$. Thus the cone
of classes $\alpha$ with $\alpha^2>0$ decomposes into two
connected components. The connected component $\kc_X$ that
contains the K\"ahler cone $\kk_X$, i.e.\ the cone of all K\"ahler
classes, is called the positive cone. Note that for the Mukai
lattice $\widetilde H(X,\ZZ)$ the set of real $(1,1)$-classes of
positive square is connected.

The N\'eron--Severi group $\NS(X)$ is identified with
$H^{1,1}(X)\cap H^2(X,\ZZ)$ and its rank is the Picard number
$\rho(X)$. Since $X$ is projective, the intersection form on
$\NS(X)_\RR$ has signature $(1,\rho(X)-1)$. The transcendental
lattice $T(X)$ is by definition the orthogonal complement of
$\NS(X)\subset H^2(X,\ZZ)$.  Hence, $H^2(X,\QQ)=\NS(X)_\QQ\oplus
T(X)_\QQ$ which can be read as an orthogonal decomposition of
weight two rational Hodge structures (but in general not over
$\ZZ$). Note that $T(X)_\QQ$ cannot be decomposed further, it is
an irreducible Hodge structure. The ample cone is the intersection
of the K\"ahler cone $\kk_X$ with $\NS(X)_\RR$ and is spanned by
ample line bundles.

Analogously, one has the extended N\'eron--Severi group
$$\widetilde\NS(X):=\widetilde H^{1,1}(X)\cap \widetilde
H(X,\ZZ)=\NS(X)\oplus (H^0\oplus H^4)(X,\ZZ).$$

Note that $\widetilde \NS(X)$ is simply the lattice of all
algebraic classes. More precisely, $\widetilde\NS(X)$ can be seen
as the image of the cycle map $\CH^*(X)\to H^*(X,\ZZ)$ or the set
of all Mukai vectors $v(E)=\ch(E).\sqrt{\td(X)}=\ch(E).(1,0,1)$
with $E\in\Db(X)$. Note that the transcendental lattice in
$\widetilde H(X,\ZZ)$ coincides with $T(X)$.

\subsection{}\label{subsect:cone} The so-called $(-2)$-classes, i.e.\
integral $(1,1)$-classes $\delta$ with $\delta^2=-2$, play a
central role in the classical theory as well as in the modern part
related to derived categories and Chow groups.

Classically, one considers the set $\Delta_X$ of $(-2)$-classes in
$\NS(X)$. E.g.\ every smooth rational curve $\PP^1\cong C\subset
X$ defines by adjunction a $(-2)$-class, hence $C$ is called a
$(-2)$-curve. Examples of $(-2)$-classes in the extended
N\'eron--Severi lattice $\widetilde \NS(X)$ are provided by the
Mukai vector $v(E)$ of spherical objects $E\in\Db(X)$ (see
\ref{subsect:FM}  and \ref{subsect:Sph}). Note that
$v(\ko_C)\ne[C]$, but $v(\ko_C(-1))=[C]$. For later use we
introduce $\widetilde\Delta_X$ as the set of $(-2)$-classes in
$\widetilde\NS(X)$.

Clearly, an ample or, more generally, a K\"ahler class has
positive intersection with all effective curves and with
$(-2)$-curves in particular. Conversely, one knows that every
class $\alpha\in\kc_X$ with $(\alpha.C)>0$ for all $(-2)$-curves
is a K\"ahler class (cf.\ \cite{BHPV}).

To any $(-2)$-class $\delta$ one associates the reflection
$s_\delta:\alpha\mapsto \alpha+(\alpha.\delta)\delta$ which is an
orthogonal transformation of the lattice also  preserving the
Hodge structure. The Weyl group is by definition the subgroup of
the orthogonal group generated by reflections $s_\delta$. So one
has two groups $W_X\subset \OO(H^2(X,\ZZ))$ and $\widetilde
W_X\subset\OO(\widetilde H(X,\ZZ))$.

The union of hyperplanes $\bigcup_{\delta\in\Delta_X}\delta^\perp$
is locally finite in the interior of $\kc_X$ and endows $\kc_X$
with a chamber structure. The Weyl group $W_X$ acts simply
transitively on the set of chambers and the K\"ahler cone is one
of the chambers. The action of $W_X$ on $\NS(X)_\RR\cap\kc_X$ can
be studied analogously. It can also be shown that reflections
$s_{[C]}$ with $C\subset X$ smooth rational curves generate $W_X$.

Another part of the Global Torelli theorem complementing i) in
\ref{subsect:GT} says that a non-trivial automorphism
$f\in\Aut(X)$ acts always non-trivially on $H^2(X,\ZZ)$. Moreover,
any Hodge iso\-metry of $H^2(X,\ZZ)$ preserving the positive cone
is induced by an automorphism up to the action of $W_X$. In fact,
Piatetski-Shapiro and  Shafarevich also showed that the action on
$\NS(X)$ is essentially enough to determine $f$. More precisely,
one knows that the natural homomorphism
$$\Aut(X)\to\OO(\NS(X))/W_X$$
has finite kernel and cokernel. Roughly, the kernel is finite
because an automorphism that leaves invariant a polarization is an
isometry of the underlying hyperk\"ahler structure and these
isometries  form a compact group. For the finiteness of the
cokernel note that some high power of any automorphism $f$ always
acts trivially on $T(X)$.

The extended N\'eron--Severi group plays also the role of a period
domain for the space of stability conditions on $\Db(X)$ (see
\ref{subsect:ker}). For this consider  the open set
$\kp(X)\subset\widetilde \NS(X)_\CC$ of vectors whose real and
imaginary parts span a positively oriented positive plane. Then
let $\kp_0(X)\subset\kp(X)$ be the complement of the union of all
codimension two sets $\delta^\perp$ with $\delta\in\widetilde
\NS(X)$ and $\delta^2=-2$ (or, equivalently, $\delta=v(E)$ for
some spherical object $E\in\Db(X)$ as we will explain later):
$$\kp_0(X):=\kp(X)\setminus\bigcup_{\delta\in\widetilde\Delta_X}\delta^\perp.$$

Since the signature of the intersection form on $\widetilde\NS(X)$
is $(2,\rho(X))$, the set $\kp_0(X)$ is connected. Its fundamental
group $\pi_1(\kp_0(X))$ is generated by loops around each
$\delta^\perp$ and the one induced by the natural $\CC^*$-action.



\section{Chow ring}

We now turn to the second object that can naturally be associated
with any K3 surface $X$ defined over an arbitrary field $K$, the
Chow group $\CH^*(X)$. For a separably closed field like $\bar\QQ$
or $\CC$ it is torsion free due to a theorem of Roitman
\cite{Roit} and for number fields we will simply ignore everything
that is related to the possible occurrence of torsion. The
standard reference for Chow groups is Fulton's book \cite{Fulton}.
For the interplay  between Hodge theory and Chow groups see e.g.\
\cite{Voisin}.

\subsection{} The Chow group $\CH^*(X)$ of a K3 surface (over $K$) is
the group of cycles modulo rational equivalence. Thus,
$\CH^0(X)\cong\ZZ$ (generated by $[X]$) and $\CH^1(X)=\Pic(X)$.
The interesting part is $\CH^2(X)$ which behaves differently for
$K=\bar\QQ$ and $K=\CC$. Let us begin with the following
celebrated result of Mumford \cite{Mum}.

\begin{thm}{\bf (Mumford)}
If $K=\CC$, then $\CH^2(X)$ is infinite dimensional.
\end{thm}

(A priori $\CH^2(X)$ is simply a group, so one needs to explain
what it means that $\CH^2(X)$ is infinite dimensional. A first
very weak version says that $\dim_\QQ\CH^2(X)_\QQ=\infty$. For a
more geometrical and more precise definition of infinite
dimensionality see e.g.\ \cite{Voisin}.)

For $K=\bar\QQ$ the situations is expected to be different. The
Bloch--Beilinson conjectures lead one to the following conjecture
for K3 surfaces.

\begin{conj}\label{conj:BB} If $K$ is a number field or $K=\bar\QQ$, then $\CH^2(X)_\QQ=\QQ$.
\end{conj}

So, if $X$ is a K3 surface defined over $\bar\QQ$, then one
expects $\dim_\QQ\CH^2(X)_\QQ=1$, whereas for the complex K3
surface $X_\CC$ obtained by base change from $X$ one knows
$\dim_\QQ\CH^2(X_\CC)_\QQ=\infty$. To the best of my knowledge not
a single example of a K3 surface $X$ defined over $\bar\QQ$ is
known where finite dimensionality of $\CH^2(X)_\QQ$ could be
verified.

Also note that the Picard group does not change under base change
from $\bar\QQ$ to $\CC$, i.e.\ for $X$ defined over $\bar\QQ$ one
has $\Pic(X)\cong\Pic(X_\CC)$ (see \ref{subsect:Inaba}). But over
the actual field of definition of $X$, which is a number field in
this case, the Picard group can be strictly smaller.

The central argument in Mumford's proof is that an irreducible
component of the closed subset of effective cycles in $X^n$
rationally equivalent to a given cycle must be proper, due to the
existence of a non-trivial regular two-form on $X$, and that a
countable union of those cannot cover $X^n$ if the base field is
not countable. This idea was later formalized and has led to many
more results proving non-triviality of cycles under non-vanishing
hypotheses on the non-algebraic part of the cohomology (see e.g.\
\cite{Voisin}). There is also a more arithmetic approach to
produce arbitrarily many non-trivial classes in $\CH^2(X)$ for a
complex K3 surface $X$ which proceeds via curves over finitely
generated field extensions of $\bar\QQ$ and embeddings of their
function fields into $\CC$. See e.g.\ \cite{GGP}.

The degree of a cycle induces a homomorphism $\CH^2(X)\to\ZZ$ and
its kernel $\CH^2(X)_0$ is the group of homologically (or
algebraically) trivial classes. Thus, the Bloch--Beilinson
conjecture for a K3 surface $X$ over $\bar\QQ$ says that
$\CH^2(X)_0=0$ or, equivalently, that $$\CH^*(X)\cong
\widetilde\NS(X_\CC)\,\hookrightarrow\widetilde H(X_\CC,\ZZ).$$

\subsection{}\label{subsect:BV} The main results presented in my talk were triggered
by the paper of Beauville and Voisin \cite{BV} on a certain
natural subring of $\CH^*(X)$. They show in particular that for a
complex K3 surface $X$ there is a natural class $c_X\in\CH^2(X)$
of degree one with the following properties:

 i) $c_X=[x]$ for
any point $x\in X$ contained in a (singular) rational curve
$C\subset X$.

 ii) $c_1(L)^2\in\ZZ c_X$ for any $L\in\Pic(X)$.

iii) $c_2(X)=24 c_X$.

\medskip

Let us introduce $$R(X):=\CH^0(X)\oplus \CH^1(X)\oplus\ZZ c_X.$$
Then ii) shows that $R(X)$ is a subring of $\CH^*(X)$. A different
way of expressing ii) and iii) together is to say that for any
$L\in\Pic(X)$ the Mukai vector $v^\CH(L)=\ch(L)\sqrt{\td(X)}$ is
contained in $R(X)$ (see \ref{subsect:Muk}). It will be in this
form that the results of Beauville and Voisin can be generalized
in a very natural form to the derived context (Theorem
\ref{thm:sph}).

Note that the cycle map induces an isomorphism
$R(X)\cong\widetilde\NS(X)$ and that for a K3 surface $X$ over
$\bar\QQ$ the Bloch--Beilinson conjecture can be expressed by
saying that base change yields an isomorphism $\CH^*(X)\cong
R(X_\CC)$.

So, the natural filtration $\CH^*(X)_0\subset \CH^*(X)$ (see also
below) with quotient $\widetilde\NS(X)$ admits a split given by
$R(X)$. This can be written as $\CH^*(X)=R(X)\oplus\CH^*(X)_0$ and
seems to be  a special feature of K3 surfaces and
higher-dimensional symplectic varieties. E.g.\ in \cite{BeauHK} it
was conjectured that any relation between $c_1(L_i)$ of line
bundles $L_i$ on an irreducible symplectic variety $X$ in $H^*(X)$
also holds in $\CH^*(X)$. The conjecture was completed to also
incorporate Chern classes of $X$ and proved for low-dimensional
Hilbert schemes of K3 surfaces by Voisin in \cite{VoiHK}. See also
the more recent thesis by Ferretti \cite{Fer} which deals with
double EPW sextics, which are special deformations
four-dimensional Hilbert schemes.

\subsection{}\label{subsect:Bloch} The Bloch--Beilinson
conjectures also predict for  smooth projective varieties $X$ the
existence of a functorial filtration
$$0=F^{p+1}\CH^p(X)\subset F^p\CH^p(X)\subset\ldots\subset
F^1\CH^p(X)\subset F^0\CH^p(X)$$ whose first step $F^1$ is simply
the kernel of the cycle map. Natural candidates for such a
filtration were studied e.g.\ by Green, Griffiths, Jannsen, Lewis,
Murre, and S.\ Saito (see \cite{GG} and the references therein).

For a surface $X$ the interesting part of this filtration is
$0\subset\ker({\rm alb}_X)\subset \CH^2(X)_0\subset\CH^2(X)$. Here
${\rm alb}_X:\CH^2(X)_0\to{\rm Alb}(X)$ denotes the Albanese map.

A cycle $\Gamma\in\CH^2(X\times X)$ naturally acts on cohomology
and on the Chow group. We write $[\Gamma]^{i,0}_*$ for the induced
endomorphism of $H^0(X,\Omega_X^i)$ and $[\Gamma]_*$ for the
action on $\CH^2(X)$. The latter  respects the natural filtration
$\ker({\rm alb}_X)\subset \CH^2(X)_0\subset\CH^2(X)$ and thus
induces an endomorphism ${\rm gr}[\Gamma]_*$ of the graded object
$\ker({\rm alb}_X)\oplus{\rm Alb}(X)\oplus\ZZ$.

The following is  also a consequence of Bloch's conjecture, see
\cite{Bl} or \cite[Ch.\ 11]{Voisin}, not completely unrelated to
Conjecture \ref{conj:BB}.

\begin{conj}\label{conj:BBFiltr}
$[\Gamma]_*^{2,0}=0$ if and only if ${\rm gr}[\Gamma]_*=0$ on
$\ker({\rm alb}_X)$.
\end{conj}

It is known that this conjecture is implied by the
Bloch--Beilinson conjecture for $X\times X$ when $X$ and $\Gamma$
are defined over $\bar\QQ$. But otherwise, very little is known
about it. Note that the analogous statement $[\Gamma]_*^{1,0}=0$
if and only if ${\rm gr}[\Gamma]_*=0$ on ${\rm Alb}(X)$ holds true
by definition of the Albanese.

For K3 surfaces the Albanese map is trivial and so the
Bloch--Beilinson filtration for K3 surfaces is simply $0\subset
\ker({\rm alb}_X)=\CH^2(X)_0\subset\CH^2(X)$. In particular
Conjecture \ref{conj:BBFiltr} for a K3 surface becomes:
$[\Gamma]_*^{2,0}=0$ if and only if ${\rm gr}[\Gamma]_*=0$ on
$\CH^2(X)_0$. In this form the conjecture seems out of reach for
the time being, but the following special case seems more
accessible and we will explain in Section \ref{sect:Aut} to what
extend derived techniques can be useful to answer it.

\begin{conj}\label{conj:BlochAut}
Let $f\in\Aut(X)$ be a symplectomorphism of a complex projective
K3 surface $X$, i.e.\ $f^*={\rm id}$ on $H^{2,0}(X)$. Then
$f^*={\rm id}$ on $\CH^2(X)$.
\end{conj}

\begin{remark}
Note that the converse is true: If $f\in\Aut(X)$ acts as ${\rm
id}$ on $\CH^2(X)$, then $f$ is a symplectomorphism. This is
reminiscent of a consequence of the Global Torelli theorem which
for a complex projective K3 surface $X$ states:

$\bullet$ $f={\rm id}$ if and only if $f^*={\rm id}$ on the Chow
ring(!) $\CH^*(X)$.
\end{remark}

\section{Derived category}

The Chow group $\CH^*(X)$ is the space of cycles divided by
rational equi\-valence. Equivalently, one could take the abelian
or derived category of coherent sheaves on $X$ and pass to the
Grothendieck K-groups. It turns out that considering the more
rigid structure of a category that lies behind the Chow group can
lead to new insight. See \cite{FM} for a general introduction to
derived categories and for more references to the original
literature.

\subsection{}\label{subsect:Muk} For a K3 surface $X$ over a field $K$ the category $\coh(X)$ of
coherent sheaves on $X$ is a $K$-linear abelian category and its
\emph{bounded derived category}, denoted $\Db(X)$, is a $K$-linear
triangulated category.

If $E^\bullet$ is an object of $\Db(X)$, its \emph{Mukai vector}
$v(E^\bullet)=\sum (-1)^iv(E^i)=\sum
(-1)^iv(\kh^i(E^\bullet))\in\widetilde \NS(X)\subset\widetilde
H(X,\ZZ)$ is well defined. By abuse of notation, we will write the
Mukai vector as a map $$v:\Db(X)\to\widetilde \NS(X).$$ Since the
Chern character of a coherent sheaf and the Todd genus of $X$
exist as classes in $\CH^*(X)$, the Mukai vector with values in
$\CH^*(X)$ can also be  defined. This will be written as
$$v^\CH:\Db(X)\to\CH^*(X).$$
(It is a special feature of K3 surfaces that the Chern character
really is integral.)

Note that $\CH^*(X)$ can also be understood as the Grothendieck
K-group of the abelian category $\coh(X)$ or of the triangulated
category $\Db(X)$, i.e.\ $K(X)\cong K(\coh(X))\cong K(\Db(X))\cong
\CH^*(X)$. (In order to exclude any torsion phenomena we assume
here that $K$ is algebraically closed, i.e. $K=\CC$ or
$K=\bar\QQ$, or, alternatively, pass to the associated
$\QQ$-vector spaces.)

Clearly, the lift of a class in $\CH^*(X)$ to an object in
$\Db(X)$ is never unique. Of course, for certain classes there are
natural choices, e.g.\ $v^\CH(L)$ naturally lifts to $L$ which is
a spherical object (see below).

\subsection{}\label{subsect:FM} Due to a result of Orlov, every $K$-linear equivalence
$\Phi:\Db(X)\congpf\Db(X')$ between the derived categories of two
smooth projective varieties is a Fourier--Mukai transform, i.e.\
there exists a unique object $\ke\in\Db(X\times X')$ such that
$\Phi$ is isomorphic to the functor
$\Phi_\ke=p_*(q^*(~~)\otimes\ke)$. Here $p_*$, $q^*$, and
$\otimes$ are derived functors. It is known that if $X$ is a K3
surface also $X'$ is one.

It would be very interesting to use Orlov's result to deduce the
existence of objects in $\Db(X\times X')$ that are otherwise
difficult to describe. However, we are not aware of any
non-trivial example of a functor that can be shown to be an
equivalence, or even just fully faithful, without actually
describing it as a Fourier--Mukai transform.
\medskip

Here is a list of essentially all known (auto)equivalences for K3
surfaces:

i) Any isomorphism $f:X\congpf X'$ induces an exact equivalence
$f_*:\Db(X)\congpf\Db(X')$ with Fourier--Mukai kernel the
structure sheaf $\ko_{\Gamma_f}$ of the graph $\Gamma_f\subset
X\times X'$ of $f$.

ii) The tensor product $L\otimes(~~)$ for a line bundle
$L\in\Pic(X)$ defines an autoequivalence of $\Db(X)$ with
Fourier--Mukai kernel $\Delta_*L$.

iii) An object $E\in\Db(X)$ is called \emph{spherical} if
$\Ext^*(E,E)\cong H^*(S^2,K)$ as graded vector spaces. The
\emph{spherical twist} $$T_E:\Db(X)\congpf\Db(X)$$ associated with
it is the Fourier--Mukai equivalence whose kernel is given as the
cone of the trace map $E^*\boxtimes E\to (E^*\boxtimes
E)|_\Delta\congpf\Delta_*(E^*\otimes E)\to \ko_\Delta$. (For
example of spherical objects see \ref{subsect:Sph}.)

iv) If $X'$ is a fine projective moduli space of stable sheaves
and $\dim(X')=2$, then the universal family $\ke$ on $X\times X'$
(unique up to a twist with a line bundle on $X'$) can be taken as
the kernel of an equivalence $\Db(X)\congpf\Db(X')$.

\subsection{} Writing an equivalence as a Fourier--Mukai transform
allows one to associate directly to any autoequivalence
$\Phi:\Db(X)\congpf\Db(X)$ of a complex K3 surface $X$ an
isomorphism
$$\Phi^H:\widetilde H(X,\ZZ)\congpf\widetilde H(X,\ZZ)$$ which in
terms of the Fourier--Mukai kernel $\ke$ is given by
$\alpha\mapsto p_*(q^*\alpha.v(\ke))$. As was observed by Mukai,
this isomorphism is defined over $\ZZ$ and not only over $\QQ$.
Moreover, it preserves the Mukai pairing and the natural weight
two Hodge structure, i.e.\ it is an integral Hodge isometry of
$\widetilde H(X,\ZZ)$. As above, $v(\ke)$ denotes the Mukai vector
$v(\ke)=\ch(\ke)\sqrt{{\rm td}(X\times X)}$.

Clearly, the latter makes also sense in $\CH^*(X\times X)$ and so
one can as well associate to the equivalence $\Phi$ a group
automorphism
$$\Phi^\CH:\CH^*(X)\congpf\CH^*(X).$$

The reason why the usual Chern character is replaced by the Mukai
vector is the Grothendieck--Riemann--Roch formula. With this
definition of $\Phi^H$ and $\Phi^\CH$ one finds that
$\Phi^H(v(E))=v(\Phi(E))$ and $\Phi^\CH(v^\CH(E))=v^\CH(\Phi(E))$
for all $E\in\Db(X)$.

Note that $\Phi^H$ and $\Phi^\CH$ do not preserve, in general,
neither the multiplicative  structure nor the grading of
$\widetilde H(X,\ZZ)$ resp.\ $\CH^*(X)$.

The derived category $\Db(X)$ is difficult to describe in concrete
terms. Its group of autoequivalences, however, seems more
accessible. So let $\Aut(\Db(X))$ denote the group of all
$K$-linear exact equivalences $\Phi:\Db(X)\congpf\Db(X)$ up to
isomorphism. Then $\Phi\mapsto\Phi^H$ and $\Phi\mapsto \Phi^\CH$
define the two representations
$$\rho^H:\Aut(\Db(X))\to\OO(\widetilde H(X,\ZZ))\phantom{MM}
{\rm
and}\phantom{MM}\rho^\CH:\Aut(\Db(X))\to\Aut(\CH^*(X)).$$

Here, $\OO(\widetilde H(X,\ZZ))$ is the group of all integral
Hodge isometries of the weight two Hodge structure defined on the
Mukai lattice $\widetilde H(X,\ZZ)$  and $\Aut(\CH^*(X))$ denotes
simply the group of all automorphisms of the additive group
$\CH^*(X)$.

Although $\CH^*(X)$  is a much bigger group than $\widetilde
H(X,\ZZ)$, at least over $K=\CC$, both representations carry
essentially the same information. More precisely one can prove
(see \cite{HSp}):

\begin{thm}\label{thm:sameker}
$\ker(\rho^H)=\ker(\rho^\CH)$.
\end{thm}

In the following we will explain what is known about this kernel
and the images of the representations $\rho^H$ and $\rho^\CH$.

\subsection{}\label{subsect:ker}
Due to the existence of the many spherical objects in $\Db(X)$ and
their associated spherical twists, the kernel
$\ker(\rho^H)=\ker(\rho^\CH)$ has a rather intriguing structure.
Let us be a bit more precise: If $E\in\Db(X)$ is spherical, then
$T_E^H$ is the reflection $s_\delta$ in the hyperplane orthogonal
to $\delta:=v(E)$. Hence, the square $T_E^2$ is an element in
$\ker(\rho^H)$ which is easily shown to be non-trivial.

Due to the existence of the many spherical objects on any K3
surface, e.g.\ all line bundles are spherical, and the complicated
relations between them, the group generated by all $T_E^2$ is a
very interesting object. In fact, conjecturally $\ker(\rho^H)$ is
generated by the $T_E^2$'s and the double shift. This and the
expected relations between the spherical twists are expressed by
the following conjecture of Bridgeland \cite{Br}.

\begin{conj}\label{conj:Bridg}
$\ker(\rho^H)=\ker(\rho^\CH)\cong \pi_1(\kp_0(X))$.
\end{conj}

For the definition of $\kp_0(X)$ see \ref{subsect:cone}. The
fundamental group of $\kp_0(X)$ is generated by loops around each
$\delta^\perp$ and the generator of $\pi_1(\kp(X))\cong\ZZ$. The
latter is naturally lifted to the auto\-equi\-valence given by the
double shift $E\mapsto E[2]$.

Since each $(-2)$-vector $\delta$ can be written as $\delta=v(E)$
for some spherical object, one can lift the loop around
$\delta^\perp$ to $T^2_E$. However, the spherical object $E$ is by
no means unique. Just choose any other spherical object $F$ and
consider $T_F^2(E)$ which has the same Mukai vector as $E$. Even
for a Mukai vector $v=(r,\ell,s)$ with $r>0$ there is in general
more than one spherical bundle(!) $E$ with $v(E)=v$ (see
\ref{subsect:Sph}).

Nevertheless, Bridgeland does construct a group homomorphism
$$\pi_1(\kp_0(X))\to\ker(\rho^H)\subset \Aut(\Db(X)).$$  The injectivity
of this map is equivalent to the simply connectedness of the
distinguished component $\Sigma(X)\subset{\rm Stab}(X)$ of
stability conditions considered by Bridgeland. If $\Sigma(X)$ is
the only connected component, then the surjectivity would follow.

Note that, although $\ker(\rho^H)$ is by definition not visible on
$\widetilde H(X,\ZZ)$ and by Theorem \ref{thm:sameker} also not on
$\CH^*(X)$, it still seems to be governed by the Hodge structure
of $\widetilde H(X,\ZZ)$. Is this in any way reminiscent of the
Bloch conjecture (see \ref{subsect:Bloch})?

\subsection{}
On the other hand, the image of $\rho^H$ is well understood which
is (see \cite{HMS}):

\begin{thm}\label{thm:HMS}
The image of $\rho^H:\Aut(\Db(X))\to\OO(\widetilde H(X,\ZZ))$ is
the group $\OO_+(\widetilde H(X,\ZZ))$ of all Hodge isometries
leaving invariant the natural orientation of the space of positive
directions.
\end{thm}

Recall that the Mukai pairing has signature $(4,20)$. The classes
${\rm Re}(\sigma),{\rm Im}(\sigma), 1-\omega^2/2,\omega$, where
$0\ne\sigma\in H^{2,0}(X)$ and $\omega\in\kk_X$ an ample class,
span a real subspace $V$ of dimension four which is positive
definite with respect to the Mukai pairing. Using orthogonal
projection, the orientations of $V$ and $\Phi^H(V)$ can be
compared.

To show that  ${\rm Im}(\rho^H)$ has at most index two in
$\OO(\widetilde H(X,\ZZ))$ uses techniques of Mukai and Orlov and
was observed by Hosono, Lian, Oguiso, Yau \cite{HLOY} and Ploog.
As it turned out, the difficult part is to prove that the index is
exactly two. This was predicted by Szendr\H{o}i, based on
considerations in mirror symmetry, and recently proved in a joint
work with Macr\`i and Stellari \cite{HMS}.

Let us now turn to the image of $\rho^\CH$. The only additional
structure the Chow group $\CH^*(X)$ seems to have is the subring
$R(X)\subset\CH^*(X)$ (see \ref{subsect:BV}).  And indeed, this
subring is preserved under derived equivalences (see \cite{HSp}):

\begin{thm}\label{thm:Rpres}
If $\rho(X)\geq2$ and $\Phi\in\Aut(\Db(X))$, then $\Phi^H$
preserves the subring $R(X)\subset\CH^*(X)$.
\end{thm}

In other words, autoequivalences (and in fact equivalences)
respect the direct sum decomposition
$\CH^*(X)=R(X)\oplus\CH^*(X)_0$ (see \ref{subsect:BV}).

The assumption on the Picard rank should eventually be removed,
but as for questions concerning potential density of rational
points the Picard rank one case is indeed more complicated.

Clearly, the action of $\Phi^\CH$ on $R(X)$ can be completely
recovered from the action of $\Phi^H$ on $\widetilde\NS(X)$. On
the other hand, according to the Bloch conjecture (see
\ref{subsect:Bloch}) the action of $\Phi^\CH$ on $\CH^*(X)_0$
should be governed by the action of $\Phi^H$ on the transcendental
part $T(X)$. Note that for $K=\bar\QQ$ one expects $\CH^*(X)_0=0$,
so nothing interesting can be expected in this case. However, for
$K=\CC$ well-known arguments show that $\Phi^H\ne{\rm id}$ on
$T(X)$ implies $\Phi^\CH\ne{\rm id}$ on $\CH^*(X)_0$ (see
\cite{Voisin}). As usual, it is the converse that is much harder
to come by. Let us nevertheless rephrase the Bloch conjecture once
more for this case.

\begin{conj}
Suppose $\Phi^H={\rm id}$ on $T(X)$. Then $\Phi^\CH={\rm id}$ on
$\CH^*(X)_0$.
\end{conj}

By  Theorem \ref{thm:sameker} one has $\Phi^\CH={\rm id}$ under
the stronger assumption $\Phi^H={\rm id}$ not only on $T(X)$ but
on all of $\widetilde H(X,\ZZ)$. The special case of $\Phi=f_*$
will be discussed in more detail in Section \ref{sect:Aut}

Note that even if the conjecture can be proved we would still not
know how to describe the image of $\rho^\CH$. It seems, $\CH^*(X)$
has just not enough structure that could be used to determine
explicitly which automorphisms are induced by derived
equivalences.


\section{Chern classes of spherical objects}

It has become clear that spherical objects and the associated
spherical twists play a central role in the description of
$\Aut(\Db(X))$. Together with automorphisms of $X$ itself and
orthogonal transformations of $\widetilde H$ coming from universal
families of stable bundles, they determine the action of
$\Aut(\Db(X))$ on $\widetilde H(X,\ZZ)$. The description of the
kernel of $\rho^\CH$ should only involve squares of spherical
twists by Conjecture \ref{conj:Bridg}.

\subsection{}\label{subsect:Sph} It is time to give more examples of spherical
objects.

 i) Every line bundle $L\in\Pic(X)$ is a spherical object
in $\Db(X)$ with Mukai vector $v=(1,\ell,\ell^2/2+1)$ where
$\ell=c_1(L)$. Note that the spherical twist $T_L$ has nothing to
do with the equivalence given by the tensor product with $L$. Also
the relation between $T_L$ and e.g.\ $T_{L^2}$ is  not obvious.

ii) If $C\subset X$ is a smooth irreducible rational curve, then
all $\ko_C(i)$ are spherical objects with Mukai vector
$v=(0,[C],i+1)$. The spherical twist $T_{\ko_C(-1)}$ induces the
reflection $s_{[C]}$ on $\widetilde H(X,\ZZ)$, an element of the
Weyl group $W_X$.

 iii) Any simple vector bundle
$E$ which is also rigid, i.e.\ $\Ext^1(E,E)=0$, is spherical. This
generalizes i).  Note that rigid torsion free sheaves are
automatically locally free (see \cite{Muk}). Let
$v=(r,\ell,s)\in\widetilde NS(X)$ be a $(-2)$-class with $r>0$ and
$H$ be a fixed polarization. Then due to a result of Mukai there
exists a unique rigid bundle $E$ with $v(E)=v$ which is slope
stable with respect to $H$ (see \cite{HL}). However, varying $H$
usually leads to (finitely many) different spherical bundles
realizing $v$. They should be considered as non-separated points
in the moduli space of simple bundles (on deformations of $X$).
This can be made precise by saying that for two different
spherical bundles $E_1$ and $E_2$ with $v(E_1)=v(E_2)$ there
always exists a non-trivial homomorphism $E_1\to E_2$.

\subsection{} The Mukai vector $v(E)$ of a spherical object $E\in\Db(X)$ is
an integral $(1,1)$-class of square $-2$ and every such class can
be lifted to a spherical object. For the Mukai vectors in
$\CH^*(X)$ we have the following (see \cite{HSp}):

\begin{thm}\label{thm:sph}
If $\rho(X)\geq2$ and $E\in\Db(X)$ is spherical, then $v^\CH(E)\in
R(X)$.
\end{thm}

In particular, two non-isomorphic spherical bundles realizing the
same Mukai vector in $\widetilde H(X,\ZZ)$ are also not
distinguished by their Mukai vectors in $\CH^*(X)$. Again, the
result should hold without the assumption on the Picard group.

This theorem is first proved for spherical bundles by using
Lazarsfeld's technique to show that primitive ample curves on K3
surfaces are Brill--Noether general \cite{Laz} and the
Bogomolov--Mumford theorem on the existence of rational curves in
ample linear systems \cite{MM} (which is also at the core of
\cite{BV}). Then one uses Theorem \ref{thm:sameker} to generalize
this to spherical objects realizing the Mukai vector of a
spherical bundle. For this step one observes that knowing the
Mukai vector of the Fourier--Mukai kernel of $T_E$ in
$\CH^*(X\times X)$ allows one to determine $v^\CH(E)$.

Actually Theorem \ref{thm:sph} is proved first and  Theorem
\ref{thm:Rpres} is a consequence of it, Indeed, if
$\Phi:\Db(X)\congpf\Db(X)$ is an equivalence, then for a spherical
object  $E\in\Db(X)$ the image $\Phi(E)$ is again spherical. Since
$v^\CH(\Phi(E))=\Phi^\CH(v^\CH(E))$, Theorem \ref{thm:sph} shows
that $\Phi^\CH$ sends Mukai vectors of spherical objects, in
particular of line bundles, to classes in $R(X)$. Clearly, $R(X)$
is generated as a group by the $v^\CH(L)$ with $L\in\Pic(X)$ which
then proves Theorem \ref{thm:Rpres}.

\subsection{}\label{subsect:Inaba} The true reason behind Theorem \ref{thm:sph} and in
fact behind most of the results in \cite{BV} is the general
philosophy that every rigid geometric object on a variety $X$ is
already defined over the smallest algebraically closed field of
definition of $X$. This is then combined with the Bloch--Beilinson
conjecture which for $X$ defined over $\bar\QQ$ predicts that
$R(X_\CC)=\CH^*(X)$.

To make this more precise consider a K3 surface $X$ over $\bar\QQ$
and the associated complex K3 surface $X_\CC$. An object
$E\in\Db(X_\CC)$ is defined over $\bar\QQ$ if there exists an
object $F\in\Db(X)$ such that its base-change to $X_\CC$ is
isomorphic to $E$. We write this as $E\cong F_\CC$.

The pull-back yields an injection of rings
$\CH^*(X)\,\hookrightarrow \CH^*(X_\CC)$ and if $E\in\Db(X_\CC)$
is defined over $\bar\QQ$ its Mukai vector $v^\CH(E)$ is contained
in the image of this map. Now, if we can show that
$\CH^*(X)=R(X_\CC)$, then the Mukai vector of every
$E\in\Db(X_\CC)$ defined over $\bar\QQ$ is contained in
$R(X_\CC)$.

Eventually one observes that spherical objects on $X_\CC$ are
defined over $\bar\QQ$. For line bundles $L\in\Pic(X_\CC)$ this is
well-known, i.e.\ $\Pic(X)\cong\Pic(X_\CC)$. Indeed, the Picard
functor is defined over $\bar\QQ$ (or in fact over the field of
definition of $X$) and therefore the set of connected components
of the Picard scheme does not change under base change. The Picard
scheme of a K3 surface is zero-dimensional, a connected component
consists of one closed point and, therefore, base change
identifies the set of closed points. For the algebraically closed
field $\bar\QQ$ the set of closed points of the Picard scheme of
$X$ is the Picard group of $X$ which thus does not get bigger
under base change e.g.\ to $\CC$.

For general spherical objects in $\Db(X_\CC)$ the proof uses
results of Inaba and Lieblich (see e.g.\ \cite{Inaba1}) on the
representability of the functor of complexes (with vanishing
negative $\Ext$'s) by an algebraic space. This is technically more
involved, but the underlying idea is just the same as for the case
of line bundles.


\section{Automorphisms acting on the Chow ring}\label{sect:Aut}

We come back to the question raised as Conjecture
\ref{conj:BlochAut}. So suppose $f\in\Aut(X)$ is an automorphism
of a complex projective K3 surface $X$ with $f^*\sigma=\sigma$
where $\sigma$ is a trivializing section of the canonical bundle
$\omega_X$. In other words, the Hodge isometry $f^*$ of
$H^2(X,\ZZ)$ (or of $\widetilde H(X,\ZZ)$) is the identity on
$H^{0,2}(X)=\widetilde H^{0,2}(X)$ or, equivalently, on the
transcendental lattice $T(X)$. What can we say about the action
induced by $f$ on $\CH^2(X)$? Obviously, the question makes sense
for K3 surfaces defined over other fields, e.g.\ for $\bar\QQ$,
but $\CC$ is the most interesting case (at least in characteristic
zero) and for $\bar\QQ$ the answer should be without any interest
due to the Bloch--Beilinson conjecture.

In this section we will explain that the techniques of the earlier
sections and of \cite{HSp} can be combined with results of Kneser
on the orthogonal group of lattices to  prove Conjecture
\ref{conj:BlochAut} under some additional assumptions on the
Picard group of $X$.

\subsection{}\label{subsect:vGGS} Suppose $f\in\Aut(X)$ is a
non-trivial  symplectomorphism, i.e.\ $f^*\sigma=\sigma$. If $f$
has finite order $n$, then $n=2,\ldots,7$, or $8$. This is a
result due to Nikulin \cite{Nik} and follows from the holomorphic
fixed point formula (see \cite{Mukai2}). Moreover, in this case
$f$ has only finitely many fixed points, all isolated, and
depending on $n$ the number of fixed point is $8,6,4,4,2,3$, resp.
$2$. The minimal resolution of the quotient $Y\to \bar
X:=X/\langle f\rangle$ yields again a K3 surface $Y$. Thus, for
symplecto\-morphisms of finite order Conjecture
\ref{conj:BlochAut} is equivalent to the bijectivity of the
natural map $\CH^2(Y)_\QQ\to\CH^2(X)_\QQ$. Due to a result of
Nikulin the action of a symplectomorphism $f$ of finite order on
$H^2(X,\ZZ)$ is as an abstract lattice automorphism independent of
$f$ and depends only on the order. For prime order $2,3,5$, and
$7$ it was explicitly described  and studied in \cite{vGS,GS}.
E.g.\ for a symplectic involution the fixed part in $H^2(X,\ZZ)$
has rank $14$. The moduli space of K3 surfaces $X$ endowed with
with a symplectic involution is of dimension $11$ and the Picard
group of the generic member contains $E_8(-2)$ as a primitive
sublattice of corank one.

Explicit examples of symplectomorphisms are easy to construct.
E.g.\  $(x_0:x_1:x_2:x_3)\mapsto (-x_0:-x_1:x_2:x_3)$ defines a
symplectic involution on the Fermat quartic $X_0\subset\PP^3$. On
an elliptic K3 surface with two sections one can use fibrewise
addition to produce symplectomorphisms.

\subsection{} The orthogonal group of a unimodular lattice
$\Lambda$ has been investigated in detail by Wall in \cite{Wall}.
Subsequently, there have been many attempts to generalize some of
his results to non-unimodular lattices. Of course, often new
techniques are required in the more general setting and some of
the results do not hold any longer.

The article of Kneser \cite{Kneser} turned out to be particularly
relevant for our purpose. Before we can state Kneser's result we
need to recall a few notions. First, the Witt index of a lattice
$\Lambda$ is the maximal dimension of an isotropic subspace in
$\Lambda_\RR$. So, if $\Lambda$ is non-degenerate of  signature
$(p,q)$, then the Witt index is $\min\{p,q\}$. The $p$-rank
$\rk_p(\Lambda)$ of $\Lambda$ is the maximal rank of a sublattice
$\Lambda'\subset\Lambda$ whose discriminant is not divisible by
$p$.

Recall that every orthogonal transformation of the real vector
space $\Lambda_\RR$ can be written as a composition of
reflections. The spinor norm of a reflection with respect to a
vector $v\in\Lambda_\RR$ is defined as $-(v,v)/2$ in
$\RR^*/{\RR^*}^2$. In particular, a reflection $s_\delta$ for a
$(-2)$-class $\delta\in\Lambda$ has trivial spinor norm. The
spinor norm for reflections is extended multiplicatively to a
homomorphism $\OO(\Lambda)\to\{\pm1\}$.

The following is a classical result due to Kneser, motivated by
work of Ebeling, which does not seem  widely known.

\begin{thm} Let $\Lambda$ be an even non-degenerate  lattice of Witt index at
least two such that $\Lambda$ represents $-2$. Suppose
$\rk_2(\Lambda)\geq 6$ and $\rk_3(\Lambda)\geq 5$. Then every
$g\in{\rm SO}(\Lambda)$ with $g={\rm id}$ on $\Lambda^*/\Lambda$
and trivial spinor norm can be written as a composition of an even
number of reflections $\prod s_{\delta_i}$ with $(-2)$-classes
$\delta_i\in\Lambda$.
\end{thm}

By using that a $(-2)$-reflection has determinant $-1$ and trivial
spinor norm and discriminant, Kneser's result can be rephrased as
follows: Under the above conditions on $\Lambda$ the Weyl group
$W_\Lambda$ of $\Lambda$ is given by
\begin{equation}\label{eqn:W}W_\Lambda=\ker\left(\OO(\Lambda)\to
\{\pm1\}\times\OO(\Lambda^*/\Lambda)\right).\end{equation}

The assumption on $\rk_2$ and $\rk_3$ can be replaced by assuming
that the reduction mod $2$ resp.\ $3$ are not of a very particular
type. E.g.\ for $p=2$ one has to exclude the case $\bar x_1\bar
x_2$, $\bar x_1\bar x_2+\bar x_3^2$, and $\bar x_1\bar x_2+\bar
x_3\bar x_4+\bar x_5^2$. See \cite{Kneser} or details.

\subsection{} Kneser's result can never be applied to the N\'eron--Severi
lattice $\NS(X)$ of a K3 surface $X$, because its Witt index is
one. But the extended N\'eron--Severi lattice $\widetilde
\NS(X)\cong\NS(X)\oplus U$ has Witt index two. The conditions on
$\rk_2$ and $\rk_3$ for $\widetilde\NS(X)$ become
$\rk_2(\NS(X))\geq 4$ and $\rk_3(\NS(X))\geq 3$. This leads to the
main result of this section.

\begin{thm}\label{thm:Kneserappl}
Suppose $\rk_2(\NS(X))\geq 4$ and $\rk_3(\NS(X))\geq3$. Then any
symplectomorphism $f\in\Aut(X)$ acts trivially on $\CH^2(X)$.
\end{thm}

\begin{proof}
First note that the discriminant of an orthogonal transformation
of a unimodular lattice is always trivial and that the
discriminant groups of $\NS(X)$ and $T(X)$ are naturally
identified. Since a symplectomorphism acts as ${\rm id}$ on
$T(X)$, its discriminant on $\NS(X)$ is also trivial. Note that a
$(-2)$-reflection $s_\delta$ has also trivial discriminant and
spinor norm $1$. Its determinant is $-1$.

Let now $\delta_0:=(1,0,-1)$, which is a class of square
$\delta_0^2=2$ (and not $-2$). So the induced reflection
$s_\delta$ has spinor norm and determinant both equal to $-1$. Its
discriminant is trivial. To a symplectomorphism $f$ we associate
the orthogonal transformation $g_f$ as follows. It is $f_*$ if the
spinor norm of $f_*$ is $1$ and $s_{\delta_0}\comp f_*$ otherwise.
Then $g_f$ has trivial spinor norm  and trivial discriminant, By
Equation (\ref{eqn:W}) this shows $g_f\in \widetilde W_X$, i.e.\
$f_*$ resp.\ $s_{\delta_0}\comp f_*$ is of the form $\prod
s_{\delta_i}$ with $(-2)$-classes $\delta_i$. Writing
$\delta_i=v(E_i)$ with spherical $E_i$ allows one to interpret the
right hand side as $\prod T^H_{E_i}$.


 Clearly, the $T^H_{E_i}$ preserve the orientation of the four
positive directions and so does $f_*$. But $s_{\delta_0}$ does
not, which proves a posteriori  that the spinor norm of $f_*$ must
always be trivial, i.e.\ $g_f=f_*$.

Thus, $f_*=\prod T^H_{E_i}$ and hence we proved that under the
assumptions on $\NS(X)$ the action of the symplectomorphism $f$ on
$\widetilde H(X,\ZZ)$ coincides with the action of the
autoequivalence $\Phi:=\prod T_{E_i}$. But by Theorem
\ref{thm:sameker}  their actions then coincide also on $\CH^*(X)$.
 To conclude, use Theorem \ref{thm:sph} which shows that the action of $\Phi$  on
$\CH^2(X)_0$ is trivial.
\end{proof}

\begin{remark}
The proof actually shows that the image of the subgroup of those
$\Phi\in\Aut(\Db(X))$ acting trivially on $T(X)$ (the `symplectic
equivalences') in $\OO(\widetilde\NS(X))$ is $\widetilde W_X$,
i.e.\ coincides with the image of the subgroup spanned by all
spherical twists $T_E$.

\end{remark}

Unfortunately, Theorem \ref{thm:Kneserappl} does not cover the
generic case of symplectomorphisms of finite order. E.g.\ the
N\'eron--Severi group of a generic K3 surface endowed  with a
symplectic involution is up to index two isomorphic to $\ZZ
\ell\oplus E_8(-2)$ (see  \cite{vGS}). Whatever the square of
$\ell$ is, the extended N\'eron--Severi lattice $\widetilde\NS(X)$
will have $\rk_2=2$ and indeed its reduction mod $2$ is of the
type $\bar x_1 \bar x_2$ explicitly excluded in Kneser's result
and its refinement alluded to above.

\begin{ex}
By a result of Morrison \cite{Mor} one knows that for Picard rank
$19$ or $20$ the N\'eron--Severi group $\NS(X)$ contains
$E_8(-1)^{\oplus 2}$ and hence the assumptions of Theorem
\ref{thm:Kneserappl} are satisfied (by far). In particular, our
result applies to the members $X_t$ of the Dwork family $\sum
x_i^4+t\prod x_i$ in $\PP^3$, so in particular to the Fermat
quartic itself. We can conclude that all symplectic automorphisms
of $X_t$ act trivially on $\CH^2(X_t)$. For the symplectic
automorphisms given by multiplication with roots of unities this
was proved by different methods already in \cite{Chatz}. To coe
back to the explicit example mentioned before: The involution of
the Fermat quartic $X_0$ given by $(x_0:x_1:x_2:x_3)\mapsto
(-x_0:-x_1:x_2:x_3)$ acts trivially on $\CH^2(X)$.
\end{ex}

 Although K3 surfaces $X$ with a symplectomorphisms $f$ and a
N\'eron--Severi group satisfying the assumptions of Theorem
\ref{thm:Kneserappl} are dense in the moduli space of  all $(X,f)$
without any condition on the N\'eron--Severi group, this is not
enough to prove Bloch's conjecture for all $(X,f)$.


\end{document}